\documentclass[12pt,a4paper]{article}

\usepackage[T1]{fontenc}
\usepackage[latin1]{inputenc}
\usepackage{lmodern}

\usepackage{amsmath}
\usepackage{amsfonts}
\usepackage{amssymb}
\usepackage{mathtools}
\usepackage{eucal}

\usepackage{cite}
\usepackage{lipsum}
\usepackage{parskip}
\usepackage{enumerate}
\usepackage{verbatim}
\usepackage{multicol}
\usepackage{fullpage}
\usepackage[pdftex]{graphicx}

\usepackage[usenames,dvipsnames,svgnames,table]{xcolor}

\usepackage{amsthm}

\newtheorem{definicao}{Definition}[section]
\newtheorem{teorema}{Theorem}[section]

\newtheorem{corolario}[teorema]{Corollary}
\newtheorem{lema}[teorema]{Lemma}
\newtheorem{rem}[teorema]{Remark}
\newcounter{exemplo}[section]
\newcommand{\exemplo}{\stepcounter{exemplo}\noindent\textbf{Example \arabic{section}.\arabic{exemplo}. }}

\newcommand{\er}{\mathbb{R}}
\newcommand{\en}{\mathbb{N}}

\newcommand{\ze}{\mathbb{Z}}

\newcommand{\Ne}{\mathrm{e}}
\newcommand{\ov}{\overline}

\newcommand{\dst}{\displaystyle}

\begin{document}
	
	\begin{center}
		{\bf {\Large Global attractivity criteria for a discrete-time Hopfield neural network model with unbounded delays via singular $M-$matrices}}
	\end{center}

	\begin{center}
		Jos\'{e} J. Oliveira$^*$, Ana Sofia Teixeira$^\ddag$
	\end{center}

	\begin{center}	
		($*$) Centro de Matem\'{a}tica (CMAT), Departamento de Matem\'{a}tica,\\
		Universidade do Minho, Campus de Gualtar, 4710-057 Braga, Portugal\\
		e-mail: jjoliveira@math.uminho.pt\\

		($\ddag$) Departamento de Matem\'{a}tica,
		Universidade do Minho, Campus de Gualtar, 4710-057 Braga, Portugal\\
		e-mail: pg49167@alunos.uminho.pt\\
		
	\end{center}

	\vskip .5cm

	\begin{abstract}
		
		In this work, we establish two global attractivity criteria for a multidimensional discrete-time non-autonomous Hopfield neural network model with infinite delays and delays in the leakage terms. The first criterion, which applies when the activation functions are bounded, is based on $M-$matrices that are not necessarily invertible. The second criterion, relevant for unbounded activation functions, requires that a related singular $M-$matrix be irreducible. We contrast our findings with existing results in the literature and present numerical simulations to illustrate the efficacy of the proposed criteria.
	\end{abstract}
	
	
	\noindent
	{\it \textbf{Mathematics Subject Classification System 2020}}: 39A10, 39A30, 39A60, 92B20.
	
	\section{Introduction}
	Neural network models have emerged as indispensable tools across diverse scientific and engineering domains, offering innovative solutions to complex challenges. Their applications span optimization \cite{Cochocki+Unbehauen}, signal processing \cite{Cochocki+Unbehauen}, image processing \cite{Aizenberg+Aizenberg+Hiltner+Moraga+Bexten}, and pattern classification \cite{zhang+lou}, where they effectively learn intricate patterns and relationships within data, resulting in highly accurate solutions. Essential attributes such as associative memory, parallel computation, pattern categorization, and fault tolerance \cite{hopfield1} underscore their wide-ranging utility.
	
	In his pioneering work, the recent Nobel Prize winner in Physics, John Hopfield  proposed a system of ordinary differential equations to describe an artificial neural network \cite{hopfield}. To better represent the inherent delays in neural signal transmission, such as finite signal propagation speeds and neuron response times, delays were incorporated into Hopfield's model.
	Marcus and Westervelt \cite{marcus-westervelt1989instability} demonstrated how such delays can destabilize the system.
	Since then, numerous studies on the stability of neural network models with delays have been published (see, for example, \cite{Aktas-Faydasicok-Arik,Berezansky+Braverman+Idels,dong-zhang-wang-2021,faria+jj,jj-jdea} and references therein).
	
	Although the stability analysis of continuous-time neural networks has received substantial attention, the discrete-time counterparts are equally important due to their relevance in computational implementations. As noted in \cite{mohamad+gopalsamy}, discrete-time systems may not fully replicate the dynamics of their continuous-time analogs, thus it is worthwhile to study the stability of discrete-time neural network models with delays.
	This paper focuses on providing sufficient conditions for the global attractivity of a general discrete-time non-autonomous Hopfield neural network model incorporating infinite distributed delays and time-varying delays in the leakage terms. Unlike many existing stability criteria based on $M$-matrices that require these matrices to be non-singular, our results allow for singular $M$-matrices, thereby broadening their applicability. As far as we are aware, the exceptions are the work of Hong and Ma \cite{hong+ma} for an autonomous discrete-time Hopfield type model with finite delays and the work of Ma, Saito, and Takeuchi \cite{ma+saito+takeuchi} for an autonomous continuous-time Hopfield type model also with finite delays.
	
	In \cite{ncube}, the author established sufficient conditions, involving non-singular $M-$matrices, for the existence and global asymptotic stability of an equilibrium in an autonomous continuous-time Hopfield neural network model with infinite distributed delays. Recently, in \cite{elmwafy+oliveira+silva}, the  main result from \cite{ncube} was generalized to a continuous-time Cohen-Grossberg neural network model, also involving non-singular $M-$matrices. For continuous-time Bidirectional Associative Memory(BAM) neural network models with delays in the leakage terms,  \cite{Berezansky+Braverman+Idels} established sufficient conditions for global asymptotic stability, again involving non-singular $M-$matrices. Subsequently, in \cite{oliveira2022}, the main result from \cite{Berezansky+Braverman+Idels} was improved, still  assuming that a certain matrix is a non-singular $M-$matrix.
	
	The stability of discrete-time neural network models has been less studied compared to continuous-time models, and stability results for discrete-time neural network models with infinite delays remain relatively scarce \cite{Chen+song+zhao+liu,jj-jdea}. The use of models with infinite delays allows the modeling of phenomena where the present evolution depends on the entire past history. It is worth noting that difference equations with finite delays can be transformed into delay-free difference equations by increasing the dimension of the realization space \cite{elaydi}, which could potentially make their study easier.
	
	In \cite{bento+jj+silva,dong-zhang-wang-2021,jj-jdea,li+zhang+liu+yang}, sufficient conditions for the global exponential stability of various discrete-time Hopfield neural network models involving non-singular $M-$matrices were established. Specifically, \cite{bento+jj+silva} provided sufficient conditions for the existence and global exponential stability of a periodic solution to a general difference equation with finite delays. This stability criterion was applied to both low-order and high-order periodic Hopfield neural network models, as well as to a BAM neural network model, with a finite delay in the leakage terms, yielding stability criteria involving non-singular $M-$matrices. In \cite{dong-zhang-wang-2021}, respectively in \cite{li+zhang+liu+yang}, a criterion for the exponential stability of a discrete-time high-order, respectively low-order, Hopfield-type neural network model with finite delays and impulses was established, again involving non-singular $M-$matrices. In \cite{jj-jdea} a global exponential stability criterion was established for a general difference equation with infinite delays and applied to discrete-time Hopfield models with both infinite delays and finite delays in the leakage terms.
	
	There are several studies on the stability of neural network models with delays on time-scales, but those that establish stability criteria involving $M-$matrices do so under the assumption that they are non-singular (see, for example, \cite{gao+wang+lin}).
	
	The main contributions of this paper are:
	\begin{description}
		\item[i.] The new technique developed to establish the global stability criteria of the model employing the properties of M-matrices and irreducible singular M-matrices, without requiring their invertibility or the construction of Lyapunov functionals,  as is often observed in the literature;
		\item[ii.] The global attractivity criterion, involving $M-$matrices non-necessarily invertible, established for the discrete-time non-autonomous Hopfield neural network model with infinite distributed and time-varying delays, system \eqref{model-principal}, assuming bounded activation functions that are not required to be differentiable or monotonic, item {\it i.} of Theorem \ref{teo-principal}. It is worth noting that in \cite{hong+ma}, the global attractivity criterion was obtained for the autonomous model  with finite delays, system \eqref{model-hong+ma}, assuming bounded, monotonic, and differentiable activation functions;
		\item[iii.] The stability criterion, involving singular irreducible $M-$matrices, for the same model \eqref{model-principal} without the assumption that the activation functions are bounded, item {\it ii.} of Theorem \ref{teo-principal}. We stress the importance of studying neural network models with unbounded activation functions \cite{sho+murata}, such as Gaussian Error Linear Unit (GELU) functions.
		To the best of our knowledge, existing stability criteria for neural network models with unbounded activation functions, when involving $M$-matrices, typically assume that these matrices are non-singular (see \cite{dong-zhang-wang-2021,jj-jdea} and references therein).
		
	\end{description}
	
	After this introduction, Section \ref{preliminaries} outlines the key notations, provides fundamental results from the theory of $M-$matrices, and introduces the discrete-time non-autonomous Hopfield neural network model considered in this work, which includes infinite distributed delays and time-varying delays in the leakage terms. Section \ref{global-attractiveness} constitutes the main contribution of the paper, where we establish the global attractiveness results. A comparison of our results with those available in the literature is also provided. In Section \ref{numerical-examples}, we include two numerical examples to illustrate the obtained results and highlight their originality in comparison with previous findings. The paper concludes with a brief discussion and summary of the main contributions.

	\section{Preliminaries and Model Description}\label{preliminaries}
	
	For $X\subseteq\er$, we denote by $X_\ze$ the set of integers numbers $X_\ze=X\cap \ze$. In this paper, for $n\in\en$, we always consider the Banach space $\er^n$ with the norm $$|\ov{x}|=|(x_1,\ldots,x_n)|=\max\big\{|x_i|: i\in[1,n]_\ze\big\}.$$
	
	Given $\ov{x}=(x_1,\ldots,x_n)\in\er^n$, we write $\ov{x}\geq0$ if $x_i\geq0$ for all $i\in[1,n]_\ze$, and  $\ov{x}>0$ if  $x_i>0$ for all $i\in[1,n]_\ze$. We write $\overline{x}\overline{y}=(x_1y_1,\ldots,x_ny_n)\in\er^n$, for any  $\overline{x}=(x_1,\dots,x_n), \overline{y}=(y_1,\ldots,y_n)\in\er^n$. For a real number $x\in\er$, we denote its integer part by $\lfloor x \rfloor$. 
	
	For $n\in\en$, we denote the set of $n\times n$ real matrices by $\er^{n\times n}$ and, as usual, we define
	$$
	Z^{n\times n}=\big\{A=[a_{ij}]\in\er^{n\times n}: a_{ij}\leq0\text{ if }i\neq j\big\}.
	$$ 
	Given $A=[a_{ij}],B=[b_{ij}]\in\er^{n\times n}$, we write $A\geq B$ if $a_{ij}\geq b_{ij}$ for all $i,j\in[1,n]_\ze$. We denote by $A^T$ the transpose of matrix $A$. 
	\begin{definicao}
		Let $n\in\en$ and $A\in Z^{n\times n}$.
		
		We say that $A$ is an $M-$\emph{matrix} if all principal minors are non-negative. 
		
		We say that $A$ is a \emph{non-singular $M-$matrix} if all principal minors are positive. 
	\end{definicao}
	
	A large number of equivalent properties exist for identifying $M-$matrices and non-singular $M-$matrices. For a detailed discussion of some of these properties, we refer the reader to Chapter 6 of \cite{Berman+Plemmons}. In this paper, we need the following result:
	\begin{teorema}\cite[Theorem 4.4.6]{Berman+Plemmons}\label{teorema-B+P-1}
		Let $A\in Z^{n\times n}$.
		
		The matrix $A$ is an $M-$matrix if and only if, for any $\ov{z}=(z_1,\ldots,z_n)\neq\ov{0}$ and taking $\ov{y}=(y_1,\ldots,y_n)=\left(A\ov{z}^T\right)^T$, there is $i\in[1,n]_\ze$ such that $z_i\neq0$ and $z_iy_i\geq0$.
	\end{teorema}
	
	\begin{definicao}\cite[Definition 2.1.2]{Berman+Plemmons}
		A square matrix $A\in\er^{n\times n}$ is \emph{reducible} if there is a permutation matrix $P$ such that 
		$$
		PAP^T=\left[\begin{array}{cc}
			B&0\\
			C&D
		\end{array}\right],
		$$
		where $B,D$ are square matrices, or if $n=1$ and $A=[0]$. Otherwise $A$ is called \emph{irreducible}.
	\end{definicao}
	
	We also need the following result:
	\begin{teorema}\cite[Theorem 6.4.16]{Berman+Plemmons}\label{teo-smmi}
		If $A$ is a singular irreducible $M-$matrix of order $n\in[2,\infty)_\ze$, then there is $\ov{d}=(d_1,\ldots,d_n)>0$ such that $A\ov{d}^T=\ov{0}$.
	\end{teorema}
	
	For a continuous function $f:\er\to\er$ such that $f(0)=0$, we denote by $f^*$ the function $f^*:[0,\infty)\to[0,\infty)$ defined by
	\begin{eqnarray}\label{def-f*}
		f^*(x)=\max\big\{f_+(x),f_+(-x),-f_-(x),-f_-(-x)\big\},\quad x\in[0,\infty),
	\end{eqnarray} 
	where the functions $f_+:\er\to\er$ and $f_-:\er\to\er$ are defined by
	$$
	f_+(x)=\left\{\begin{array}{ll}
		\max\{f(u):u\in[0,x]\},&x\geq0\\
		\\
		\max\{f(u):u\in[x,0]\},&x<0
	\end{array}\right.\,\text{and}\, f_-(x)=\left\{\begin{array}{ll}
		\min\{f(u):u\in[0,x]\},&x\geq0\\
		\\
		\min\{f(u):u\in[x,0]\},&x<0
	\end{array}\right.,
	$$
	respectively. It is straightforward to verify that $f^*$ is a non-decreasing continuous function such that $|f(x)|\leq f^*(|x|)$ for all $x\in\er$.

	For $n\in\en$, we denote by $\mathcal{B}^n$ the space of bounded functions $\ov{\psi}=(\psi_1,\ldots,\psi_n):(-\infty,0]_\ze\to\er^n$, with the norm 
	$$
	\left\|\ov{\psi}\right\|=\sup_{s\in(-\infty,0]_\ze}\left|\ov{\psi}(s)\right|.
	$$
	
	For $m\in\ze$ and $\ov{x}:\ze\to\er^n$ a vector function such that $\sup_{s\in(-\infty,0]_\ze}|\ov{x}(s)|<\infty$, we define $\ov{x}_m\in\mathcal{B}^n$ by
	$$
	\ov{x}_m(s)=\ov{x}(m+s),\quad s\in(-\infty,0]_\ze.
	$$ 
	In this paper, we investigate the following generalized discrete-time non-autonomous Hopfield neural network model that incorporates both infinite distributed delays and time-varying delays, with the latter also present in the leakage terms:
	\begin{eqnarray}\label{model-principal}
		x_i(m+1) &=&a_i(m)x_i(m-\nu_i(m))+ \displaystyle\sum_{j=1}^{n}\sum_{p=1}^P b_{ijp}(m)f_{ijp}\left(x_j(m-\tau_{ijp}(m))\right) \nonumber\\ 
		& & + \displaystyle\sum_{j=1}^{n} c_{ij} (m)  \sum_{l=0}^{\infty} \zeta_{ijl} g_{ij} (x_j(m-l)) + I_i(m),\,\, i\in[1,n]_\ze,\,m\in\en_0,
	\end{eqnarray}
	where $n,P\in\en$, $\zeta_{ijl}\in\er$, $a_i:\en_0\to(-1,1)$, $\nu_i,\tau_{ijp}:\en_0\to\en_0$, $b_{ijp},c_{ij},I_i:\en_0\to\er$, and $f_{ijp},g_{ij}:\er\to\er$ are functions. This discrete-time model generalizes several existing models in the literature, including those investigated in \cite{hong+ma,mohamad+gopalsamy,jj-jdea}.
	
	For model \eqref{model-principal} we always consider initial conditions on $\mathcal{B}^n$, i.e., for all $\sigma\in\en_0$ and $\ov{\psi}\in\mathcal{B}^n$ we consider
	\begin{eqnarray}\label{model-principal-IC}
		\ov{x}_\sigma=\ov{\psi}.
	\end{eqnarray}
	Trivially, for each $(\sigma,\ov{\psi})\in\en\times \mathcal{B}^n$, there exists a unique solution of initial value problem \eqref{model-principal}-\eqref{model-principal-IC} and we denoted it by  $\ov{x}(\cdot,\sigma,\ov{\psi})$.

	\section{Global attractiveness}\label{global-attractiveness}
	
	In this section, we give two global attractivity criteria for model \eqref{model-principal}. For this purpose, we consider the following set of hypotheses: For each $i,j\in[1,n]_\ze$ and $p\in[1,P]_\ze$,
	\begin{description}
		\item[(H1)] we have 
		$$
		a_i^+:=\sup_{m\in\en_0}|a_i(m)|<1;
		$$
		\item[(H2)] we have 
		$$
		\lim_m(m-\nu_i(m))=\lim_m(m-\tau_{ijp}(m))=\infty;
		$$
		\item[(H3)] the functions $b_{ijp}$ and $c_{ij}$ are bounded;
		\item [(H4)] we have
		$$
		\sum_{l=0}^\infty|\zeta_{ijl}|=1;
		$$
		\item[(H5)] we have
		$$
		\sum_{m=0}^\infty|I_i(m)|<\infty.
		$$
	\end{description}
	Depending on whether the activation functions $f_{ijp}$ and $g_{ij}$ are bounded, one of the following hypothesis is also assumed.
	\begin{description}
		\item[(H6)] the  functions $f_{ijp},g_{ij}:\er\to\er$ are continuous and there exist constants $F_{ijp},G_{ij}>0$ such that
		$$
		\left\{\begin{array}{ll}
			|f_{ijp}(u)|\leq F_{ijp},&\text{if }|u|>1\\
			|f_{ijp}(u)|<F_{ijp}|u|,&\text{if }|u|\leq1\text{ and }u\neq0\
		\end{array}\right.\text{ and }
		\left\{\begin{array}{ll}
			|g_{ij}(u)|\leq G_{ij},&\text{if }|u|>1\\
			|g_{ij}(u)|<G_{ij}|u|,&\text{if }|u|\leq1\text{ and }u\neq0\
		\end{array}\right.
		$$
	\end{description}
	or
	\begin{description}
		
		\item[(H6*)] the functions $f_{ijp},g_{ij}:\er\to\er$ are continuous and there exist constants $F_{ijp},G_{ij}>0$ such that
		$$
		|f_{ijp}(u)|<F_{ijp}|u|\text{ if }u\neq0\quad\text{ and }\quad
		|g_{ij}(u)|<G_{ij}|u|\text{ if }u\neq0.
		$$
	\end{description}
	\begin{rem}
		\emph{Note that, from (H6) or (H6*) we have $f_{ijp}(0)=g_{ij}(0)=0$ for all $j,i\in[1,n]_\ze$ and $p\in[1,P]_\ze$.}\\
	\end{rem}
	
	From (H3), we denote 
	\begin{eqnarray}\label{def-sup-coef}
		b_{ijp}^+=\sup_{m\in\en_0}|b_{ijp}(m)|\quad\text{and}\quad c_{ij}^+=\sup_{m\in\en_0}|c_{ij}(m)|,
	\end{eqnarray}
	and
	\begin{eqnarray}\label{def-limsup-coef}
		\hat{a}_i=\limsup_{m\to\infty}|a_i(m)|,\quad	\hat{b}_{ijp}=\limsup_{m\to\infty}|b_{ijp}(m)|\quad\text{and}\quad \hat{c}_{ij}=\limsup_{m\to\infty}|c_{ij}(m)|.
	\end{eqnarray}
	Now, consider the matrices $\mathcal{M}^+,\hat{\mathcal{M}}\in Z^{n\times n}$ defined by
	\begin{eqnarray}\label{def-matriz-M+}
		\mathcal{M}^+:=\mathcal{D}^+-\left[c_{ij}^+G_{ij}+\sum_{p=1}^Pb_{ijp}^+F_{ijp}\right]_{i,j=1}^n
	\end{eqnarray}
	where $\mathcal{D}^+:=diag(1-a_1^+,\ldots,1-a_n^+)$, and 
	\begin{eqnarray}\label{def-matriz-capeu}
		\hat{\mathcal{M}}:=\hat{\mathcal{D}}-\left[\hat{c}_{ij}G_{ij}+\sum_{p=1}^P\hat{b}_{ijp}F_{ijp}\right]_{i,j=1}^n,
	\end{eqnarray}
	where $\hat{\mathcal{D}}:=diag(1-\hat{a}_1,\ldots,1-\hat{a}_n)$, respectively. Note that we have $\mathcal{M}^+\leq\hat{\mathcal{M}}$.
	
	In order to obtain the stability criteria, we first need to prove that all solutions of \eqref{model-principal} are bounded.
	
	\begin{lema}\label{lem-solucoes-limitadas}
		Assume hypotheses (H1)-(H5).
		\begin{enumerate}
			\item If (H6) holds, then the solution $\ov{x}(\cdot,\sigma,\ov{\psi})$ of \eqref{model-principal}-\eqref{model-principal-IC} is bounded. 
			\item If (H6*) holds and $\mathcal{M}^+$ is a singular irreducible $M-$matrix, then the solution $\ov{x}(\cdot,\sigma,\ov{\psi})$ of \eqref{model-principal}-\eqref{model-principal-IC} is bounded.
		\end{enumerate}
	\end{lema}
	\begin{proof}
		Proof of item {\it i}.
		
		For $\sigma\in\en_0$ and $\ov\psi\in\mathcal{B}^n$, consider $\ov{x}(m)=(x_1(m),\ldots,x_n(m))=\ov{x}(m,\sigma,\ov\psi)$ the solution of \eqref{model-principal}-\eqref{model-principal-IC}. From \eqref{model-principal}, \eqref{def-sup-coef}, and  hypothesis (H6) we obtain   \begin{eqnarray*}
			|x_i(m+1)| \leq a_i^+|x_i(m-\nu_i(m))|+ \displaystyle\sum_{j=1}^{n}\sum_{p=1}^P b_{ijp}^+F_{ijp}  + \displaystyle\sum_{j=1}^{n} c_{ij}^+\sum_{l=0}^{\infty} |\zeta_{ijl}| G_{ij} + |I_i(m)|, 
		\end{eqnarray*}
		and by the hypotheses (H4) and (H5), we have 
		\begin{eqnarray*}
			|x_i(m+1)| \leq a_i^+|x_i(m-\nu_i(m))|+ \displaystyle\sum_{j=1}^{n}\left[\sum_{p=1}^P \left(b_{ijp}^+F_{ijp}\right)  +  c_{ij}^+ G_{ij}\right] + \sum_{l=0}^\infty|I_i(l)|,
		\end{eqnarray*}
		that is
		\begin{eqnarray}\label{eq-bounded-1}
			|x_i(m+1)| \leq a^+|x_i(m-\nu_i(m))|+C,\quad i\in[1,n]_\ze,\,\,m\in[\sigma,\infty)_\ze,
		\end{eqnarray}
		where $a^+=\dst\max_{i\in[0,n]_\ze}\left\{a_i^+\right\}$ and $C=\displaystyle\max_{i\in[0,n]_\ze}\left\{\sum_{j=1}^{n}\left[\sum_{p=1}^P \left(b_{ijp}^+F_{ijp}\right)  +  c_{ij}^+ G_{ij}\right] + \sum_{l=0}^\infty|I_i(l)|\right\}$.
		
		We claim that
		\begin{eqnarray}\label{hip-inducao}
			\|\ov{x}_{\sigma+q}\|\leq\|\ov{\psi}\|+C\sum_{k=0}^{q-1}(a^+)^k=\|\ov{\psi}\|+C\frac{1-(a^+)^q}{1-a^+},\quad q\in\en,
		\end{eqnarray}
		and, since $a^+\in[0,1)$, we conclude that the solution $\ov{x}(m)$ is bounded.
		
		In fact, for $q=1$, inequality \eqref{eq-bounded-1} implies
		$$
		|x_i(\sigma+1)|\leq a^+\|\ov\psi\|+C,\quad i\in[1,n]_\ze,
		$$
		and 
		$$
		\|\ov{x}_{\sigma+1}\|\leq\max\big\{\|\ov{\psi}\|,a^+\|\ov\psi\|+C\big\}\leq \|\ov{\psi}\|+C.
		$$
		Assuming that \eqref{hip-inducao} holds for some $q\in\en$, from \eqref{eq-bounded-1} we have
		$$
		|x_i(\sigma+q+1)|\leq a^+|x_i(\sigma+q-\nu_i(\sigma+q))|+C\leq a^+\|\ov{x}_{\sigma+q}\|+C
		$$
		and by the induction hypothesis we obtain
		\begin{eqnarray*}
			\|\ov{x}_{\sigma+q+1}\|&\leq&\max\left\{\|\ov{x}_{\sigma+q}\|,a^+\|\ov{x}_{\sigma+q}\|+C\right\}\nonumber\\
			&\leq&\max\left\{\|\ov{\psi}\|+C\sum_{k=0}^{q-1}(a^+)^k,a^+\left(\|\ov{\psi}\|+C\sum_{k=0}^{q-1}(a^+)^k\right)+C\right\}\nonumber\\
			&\leq&\max\left\{\|\ov{\psi}\|+C+C\sum_{k=1}^{q-1}(a^+)^k,a^+\|\ov{\psi}\|+C\sum_{k=1}^{q}(a^+)^k+C\right\}\nonumber\\
			&\leq&\|\ov{\psi}\|+C+C\sum_{k=1}^{q}(a^+)^k=\|\ov{\psi}\|+C\sum_{k=0}^q(a^+)^k,
		\end{eqnarray*}
		and the proof of item {\it i} is done.\\
		
		Proof of item {\it ii.}
		
		Since $\mathcal{M}^+$ is a singular irreducible $M-$matrix, Theorem \ref{teo-smmi} implies the existence of $\ov{d}=(d_1,\ldots,d_n)>0$ such that
		$\mathcal{M}^+\ov{d}^T=\ov{0}
		$, that is
		\begin{eqnarray}\label{cond-M-matrix}
			(1-a_i^+)d_i-\sum_{j=1}^nd_j\left(c_{ij}^+G_{ij}+\sum_{p=1}^Pb_{ijp}^+F_{ijp}\right)=0,\quad i\in[0,n]_\ze.
		\end{eqnarray}
		
		With the change $y_i(m)=d_i^{-1}x_i(m)$, for all $i\in[1,n]_\ze$, model \eqref{model-principal} assumes the form
		\begin{eqnarray}\label{model-principal-apos-mudanca-variavel}
			y_i(m+1) &=&a_i(m)y_i(m-\nu_i(m))+d_i^{-1} \displaystyle\sum_{j=1}^{n}\left[\sum_{p=1}^P b_{ijp}\left(m\right)f_{ijp}(d_jy_j(m-\tau_{ijp}(m)))\right. \nonumber\\ 
			& & + \displaystyle \left.c_{ij} (m)  \sum_{l=0}^{\infty} \zeta_{ijl} g_{ij} (d_jy_j(m-l))\right] + d_i^{-1}I_i(m),\,\, i\in[1,n]_\ze,\,m\in\en_0.
		\end{eqnarray}
		
		For $\sigma\in\en_0$ and $\ov{\psi}\in\mathcal{B}^n$, consider $\ov{x}(m)=(x_1(m),\ldots,x_n(m))=\ov{x}(m,\sigma,\ov{\psi})$ the solution of \eqref{model-principal}-\eqref{model-principal-IC}. 
		Suppose, by contradiction, that the solution  $\ov{x}(m)$ is unbounded. Consequently $\ov{y}(m)=(y_1(m),\ldots,y_n(m))=(d_1^{-1}x_1(m),\ldots,d_n^{-1}x_n(m))$ is also an unbounded solution of \eqref{model-principal-apos-mudanca-variavel} and there exist $i\in[1,n]_\ze$ and $m^\star\in(\sigma,\infty)_\ze$  such that
		\begin{eqnarray}\label{cond-desi-y}
			|y_i(m^\star)|>\|\ov{y}_{m^\star-1}\|+\alpha_i|I_i(m^\star-1)|,
		\end{eqnarray}
		where $\alpha_i=\left\{\begin{array}{ll}
			d_i^{-1},&\text{ if }a_i^+=0\\
			1+\frac{1-d_i}{d_ia_i^+},&\text{ if }d_i<1\text{ and }a_i^+\neq0\\
			1,&\text{ if }d_i\geq1\text{ and }a_i^+\neq0
		\end{array}\right.$.
		If not, then  we have 
		$$
		|y_i(m)|\leq\|\ov{y}_{m-1}\|+\alpha_i|I_i(m-1)|,\quad i\in[0,n]_\ze,\, m\in(\sigma,\infty)_\ze,
		$$
		thus
		$$
		|\ov{y}(m)|\leq\|\ov{y}_{m-1}\|+|\ov\alpha||\ov{I}(m-1)|,\quad m\in(\sigma,\infty)_\ze,
		$$
		where $\ov{\alpha}=(\alpha_1,\ldots,\alpha_n)$ and $\ov{I}(m)=(I_1(m),\ldots,I_n(m))$. Consequently
		$$
		|\ov{y}(m)|\leq\|\ov{y}_{\sigma}\|+|\ov{\alpha}|\sum_{l=0}^\infty|\ov{I}(l)|,\quad m\in[\sigma,\infty)_\ze,
		$$
		and, from (H5), we conclude that $\ov{y}(m)$ is bounded which is a contradiction. Thus \eqref{cond-desi-y} holds.
		
		From \eqref{model-principal-apos-mudanca-variavel} and hypotheses (H3) and (H6*) we have
		\begin{eqnarray*}
			|y_i(m^\star)| &\leq&a_i^+|y_i(m^\star-1-\nu_i(m^\star-1))|+ \displaystyle\sum_{j=1}^{n}\frac{d_j}{d_i}\left[\sum_{p=1}^P b_{ijp}^+F_{ijp}|y_j(m^\star-1-\tau_{ijp}(m^\star-1))|\right. \nonumber\\ 
			& & + \displaystyle \left.c_{ij}^+   \sum_{l=0}^{\infty} |\zeta_{ijl}| G_{ij} |y_j(m^\star-1-l)|\right] + d_i^{-1}|I_i(m^\star-1)|,
		\end{eqnarray*}
		and from (H1), (H4), and \eqref{cond-desi-y},  we obtain
		\begin{eqnarray*}
			|y_i(m^\star)| &<&a_i^+|y_i(m^\star)|-a_i^+\alpha_i|I_i(m^\star-1)|+ \displaystyle\sum_{j=1}^{n}\frac{d_j}{d_i}\left[\sum_{p=1}^P b_{ijp}^+F_{ijp}\big(|y_i(m^\star)|-\alpha_i|I_i(m^\star-1)|\big)\right. \nonumber\\ 
			& & + \displaystyle c_{ij}^+ G_{ij}\big( |y_i(m^\star)|-\alpha_i|I_i(m^\star-1)|\big)\bigg] + d_i^{-1}|I_i(m^\star-1)|,
		\end{eqnarray*}
		which is equivalent to
		\begin{eqnarray*}
			\bigg[d_i(1-a_i^+)- \lefteqn{\displaystyle\sum_{j=1}^{n}d_j\bigg(c_{ij}^+G_{ij}+\sum_{p=1}^P b_{ijp}^+F_{ijp}\bigg)\bigg]|y_i(m^\star)|}\\ &
			&<\bigg[1-d_ia_i^+\alpha_i- \alpha_i\displaystyle\sum_{j=1}^{n}d_j\bigg(c_{ij}^+G_{ij}+\sum_{p=1}^P b_{ijp}^+F_{ijp}\bigg)\bigg] |I_i(m^\star-1)|.
		\end{eqnarray*}
		From \eqref{cond-M-matrix} we conclude that
		\begin{eqnarray}\label{cond-para-abs}
			0<\bigg[1-d_ia_i^+\alpha_i- \alpha_i\displaystyle\sum_{j=1}^{n}d_j\bigg(c_{ij}^+G_{ij}+\sum_{p=1}^P b_{ijp}^+F_{ijp}\bigg)\bigg] |I_i(m^\star-1)|.
		\end{eqnarray}
		
		If $a_i^+=0$, then $\alpha_i=d_i^{-1}$ thus, from \eqref{cond-para-abs} and \eqref{cond-M-matrix} with $a_i^+=0$, we obtain
		$$
		0<\bigg[d_i- \displaystyle\sum_{j=1}^{n}d_j\bigg(c_{ij}^+G_{ij}+\sum_{p=1}^P b_{ijp}^+F_{ijp}\bigg)\bigg] |I_i(m^\star-1)|=0,
		$$
		which is not possible.
		
		If $a_i^+\neq0$ and $d_i\geq1$, then $\alpha_i=1$ thus, from \eqref{cond-M-matrix} and \eqref{cond-para-abs}, we obtain
		$$
		0<\bigg[d_i(1-a_i^+)- \displaystyle\sum_{j=1}^{n}d_j\bigg(c_{ij}^+G_{ij}+\sum_{p=1}^P b_{ijp}^+F_{ijp}\bigg)\bigg] |I_i(m^\star-1)|+(1-d_i)|I_i(m^\star-1)|\leq0,
		$$
		which is once again not possible.
		
		If $a_i^+\neq0$ and  $d_i<1$, then $\alpha_i=1+\frac{1-d_i}{d_ia_i^+}$ thus, from \eqref{cond-M-matrix} and \eqref{cond-para-abs}, we obtain
		$$
		0<\bigg[d_i(1-a_i^+)- \displaystyle\sum_{j=1}^{n}d_j\bigg(c_{ij}^+G_{ij}+\sum_{p=1}^P b_{ijp}^+F_{ijp}\bigg)\bigg] |I_i(m^\star-1)|=0,
		$$
		which is a contradiction again. Therefore the proof is concluded.
	\end{proof}
	
	In the next Lemma, recall the notations \eqref{def-f*}, \eqref{def-sup-coef}, and \eqref{def-limsup-coef}. 
	\begin{lema}\label{lem-sistema-equacoes}
		Assume hypotheses (H1)-(H5).
		\begin{enumerate}
			\item If (H6) holds, then the solution $\ov{x}(\cdot,\sigma,\ov{\psi})$ of \eqref{model-principal}-\eqref{model-principal-IC} satisfies $\dst\limsup_{m\to\infty}|x_i(m)|\leq S_i$, where $\ov{S}=(S_1,\ldots,S_n)$ satisfies the system
			\begin{eqnarray}\label{equ-aplicar-m-matriz}
				S_i=\frac{1}{1-\hat{a}_i}\sum_{j=1}^n\left(\hat{c}_{ij}g_{ij}^*(S_j)+\sum_{p=1}^P\hat{b}_{ijp}f_{ijp}^*(S_j)\right),\quad i\in[1,n]_\ze.
			\end{eqnarray}
			\item If (H6*) holds and $\mathcal{M}^+$ is a singular irreducible $M-$matrix, then the solution $\ov{x}(\cdot,\sigma,\ov{\psi})$ of \eqref{model-principal}-\eqref{model-principal-IC} satisfies $\dst\limsup_{m\to\infty}|x_i(m)|\leq S_i$, where $\ov{S}=(S_1,\ldots,S_n)$ satisfies the system
			\begin{eqnarray}\label{equ-aplicar-m-matriz-irredutivel}
				S_i=\frac{1}{1-a_i^+}\sum_{j=1}^n\left(c_{ij}^+g_{ij}^*(S_j)+\sum_{p=1}^Pb_{ijp}^+f_{ijp}^*(S_j)\right),\quad i\in[1,n]_\ze.
			\end{eqnarray}
		\end{enumerate}
	\end{lema}
	\begin{proof}
		Proof of item {\it i.}\\
		For $\sigma\in\en_0$ and $\ov\psi\in\mathcal{B}^n$, consider $\ov{x}(m)=(x_1(m),\ldots,x_n(m))=\ov{x}(m,\sigma,\ov\psi)$ the solution of \eqref{model-principal}-\eqref{model-principal-IC}. Item {\it i.} of Lemma \ref{lem-solucoes-limitadas} assures that $\ov{x}(m)$ is bounded, thus $\dst\limsup_{m\to\infty}|x_i(m)|\in\er$ for all $i\in[1,n]_\ze$.  
		
		From \eqref{model-principal} and hypotheses (H4) and (H6) we have
		\begin{eqnarray*}
			|x_i(m+1)| \leq |a_i(m)||x_i(m-\nu_i(m))|+ \displaystyle\sum_{j=1}^{n}\left(\sum_{p=1}^P |b_{ijp}(m)|F_{ijp}  +  |c_{ij}(m)| G_{ij}\right) + |I_i(m)|, 
		\end{eqnarray*}
		and from (H5) and \eqref{def-limsup-coef} we obtain
		$$
		\limsup_{m\to\infty}|x_i(m)| \leq \hat{a}_i\limsup_{m\to\infty}|x_i(m-\nu_i(m))|+ \displaystyle\sum_{j=1}^{n}\left(\sum_{p=1}^P \hat{b}_{ijp} F_{ijp}  +  \hat{c}_{ij} G_{ij}\right).
		$$
		The solution $\ov{x}(m)$ is bounded, $\dst\limsup_{m\to\infty}|x_i(m)|\geq \limsup_{m\to\infty}|x_i(m-\nu_i(m))|$ by hypothesis (H2), and  $\hat{a}_i\in[0,1)$, then
		\begin{eqnarray}\label{def-S_i0}
			\limsup_{m\to\infty}|x_i(m)| \leq \frac{1}{1-\hat{a}_i} \displaystyle\sum_{j=1}^{n}\left(\sum_{p=1}^P \hat{b}_{ijp} F_{ijp}  +  \hat{c}_{ij} G_{ij}\right)=:S_{i,0}.
		\end{eqnarray}
		Consequently, given $\varepsilon>0$ there is $m_0\in[\sigma,\infty)_\ze$ such that
		\begin{eqnarray}\label{estimativa-xi}
			|x_i(m)|\leq S_{i,0}+\varepsilon,\quad i\in[1,n]_\ze,\,m\in[m_0,\infty)_\ze.
		\end{eqnarray}
		Form (H2), for each $i,j\in[1,n]_\ze$ and $p\in[1,P]_\ze$, there is $m_{ijp}\in\en$ such that $m-\tau_{ijp}(m)\geq m_0$ for $m\in[m_{ijp},\infty)_\ze$. Defining $\tilde{m}=\max_{ijp}\{m_{ijp}\}$ we have
		\begin{eqnarray}\label{estimativa-1-dos-Ss}
			|x_j(m-\tau_{ijp}(m))|\leq S_{j,0}+\varepsilon,\quad i,j\in[1,n]_\ze, p\in[1,P]_\ze,\,m\in[\tilde{m},\infty)_\ze.
		\end{eqnarray}
		From (H4), there is $l^\star=l^\star(\varepsilon)\in\en$ such that
		\begin{eqnarray}\label{estimativa-controlo-integral}
			\sum_{l=l^\star}^\infty|\zeta_{ijl}|<\varepsilon,\quad i,j\in[1,n]_\ze.
		\end{eqnarray}
		Defining $m^\star:=m_0+l^\star$, for $m\in[ m^\star,\infty)_\ze$ and $l\in[0,l^\star-1]_\ze$ we have $m-l>m_0$ thus, from \eqref{estimativa-xi} we obtain
		\begin{eqnarray}\label{estimativa-2-dos-Ss}
			|x_i(m-l)|\leq S_{i,0}+\varepsilon,\quad i\in[1,n]_\ze,\,m\in[ m^\star,\infty)_\ze,\,l\in[0,l^\star-1]_\ze.
		\end{eqnarray}
		
		Recalling the definition of $f^*$ for a given function $f$ in \eqref{def-f*}, we know that, for each $i,j\in[1,n]_\ze$ and $p\in[1,P]_\ze$, $f^*_{ijp}$ and $g^*_{ij}$ are non-decreasing continuous functions such that 
		\begin{eqnarray*}
			|f_{ijp}(u)|\leq f^*_{ijp}(|u|),\quad\text{and}\quad|g_{ij}(u)|\leq g^*_{ij}(|u|),\quad u\in\er.
		\end{eqnarray*}
		More, from (H6) it is easy to conclude that $f^*_{ijp}$ and $g^*_{ij}$ satisfy (H6) with the same constants $F_{ijp}$ and $G_{ij}$, respectively.
		
		From model \eqref{model-principal} and the properties of $f^*_{ijp}$ and $g^*_{ij}$ described above, for all $i\in[1,n]_\ze$ and $m\in[\sigma,\infty)_\ze$ we have
		\begin{eqnarray*}
			|x_i(m+1)| &\leq&|a_i(m)||x_i(m-\nu_i(m))|+ \displaystyle\sum_{j=1}^{n}\sum_{p=1}^P |b_{ijp}(m)|f^*_{ijp}\left(|x_j(m-\tau_{ijp}(m))|\right) \nonumber\\ 
			& & + \displaystyle\sum_{j=1}^{n} |c_{ij} (m)|  \sum_{l=0}^{\infty}|\zeta_{ijl}|g^*_{ij} \left(|x_j(m-l)|\right) + |I_i(m)|.
		\end{eqnarray*}
		
		The functions $f^*_{ijp}$ and $g^*_{ij}$ are non-decreasing satisfying (H6) thus, from (H4), \eqref{estimativa-1-dos-Ss}, \eqref{estimativa-controlo-integral}, and \eqref{estimativa-2-dos-Ss}, for $m\geq\max\{\tilde{m},m^\star\}$ we have
		\begin{eqnarray}\label{etimativa-para-item-ii}
			|x_i(m+1)| &\leq&|a_i(m)||x_i(m-\nu_i(m))|+ \displaystyle\sum_{j=1}^{n}\sum_{p=1}^P |b_{ijp}(m)|f^*_{ijp}\left(S_{j,0}+\varepsilon\right) \nonumber\\ 
			& & + \displaystyle\sum_{j=1}^{n} |c_{ij} (m)|  \left(\sum_{l=0}^{l^*-1}|\zeta_{ijl}|g^*_{ij} \left(S_{j,0}+\varepsilon\right)+\sum_{l=l^*}^{\infty}|\zeta_{ijl}|g^*_{ij} \left(|x_j(m-l)|\right)\right) + |I_i(m)|\nonumber\\
			&\leq&|a_i(m)||x_i(m-\nu_i(m))|+ \displaystyle\sum_{j=1}^{n}\sum_{p=1}^P |b_{ijp}(m)|f^*_{ijp}\left(S_{j,0}+\varepsilon\right) \nonumber\\ 
			& & + \displaystyle\sum_{j=1}^{n} |c_{ij} (m)|  \left(g^*_{ij} \left(S_{j,0}+\varepsilon\right)+\varepsilon G_{ij}\right) + |I_i(m)|.
		\end{eqnarray}
		Consequently, from (H5) and \eqref{def-limsup-coef},
		$$
		\limsup_{m\to\infty}|x_i(m)| \leq \frac{1}{1-\hat{a}_i} \displaystyle\sum_{j=1}^{n}\left(\sum_{p=1}^P \hat{b}_{ijp} f^*_{ijp}\left(S_{j,0}+\varepsilon\right)  +  \hat{c}_{ij} g^*_{ij} \left(S_{j,0}+\varepsilon\right)+\hat{c}_{ij}\varepsilon G_{ij}\right).
		$$
		As  $f^*_{ijp}$ and $g^*_{ij}$ are continuous and $\varepsilon>0$ is arbitrary, we conclude that
		$$
		\limsup_{m\to\infty}|x_i(m)| \leq \frac{1}{1-\hat{a}_i} \displaystyle\sum_{j=1}^{n}\left(\sum_{p=1}^P \hat{b}_{ijp} f^*_{ijp}\left(S_{j,0}\right)  +  \hat{c}_{ij} g^*_{ij} \left(S_{j,0}\right)\right)=:S_{i,1},\quad i\in[1,n]_\ze.
		$$
		As $f^*_{ijp}(S_{j,0})\leq F_{ijp}$, $g^*_{ij}(S_{j,0})\leq G_{ij}$, and \eqref{def-S_i0}, then $S_{i,1}\leq S_{i,0}$ for all $i\in[1,n]_\ze$. Iterating the previous process, for each $i\in[1,n]_\ze$, we obtain a non-increasing sequence $\left(S_{i,q}\right)_{q\in\en}$ such that $S_{i,q}\geq0$, $\dst\limsup_{m\to\infty}|x_i(m)|\leq S_{i,q}$, and
		\begin{eqnarray}\label{relacao-entre-Ss}
			S_{i,q}=\frac{1}{1-\hat{a_i}}\sum_{j=1}^n\left(\sum_{p=1}^P \hat{b}_{ijp} f^*_{ijp}\left(S_{j,q-1}\right)  +  \hat{c}_{ij} g^*_{ij} \left(S_{j,q-1}\right)\right),\quad q\in\en.
		\end{eqnarray}
		Defining $S_i:=\dst\lim_{q\to\infty}S_{i,q}$, we have
		$$
		\limsup_{m\to\infty}|x_i(m)|\leq S_i,\quad i\in[1,n]_\ze.
		$$
		Finally, from \eqref{relacao-entre-Ss} and the continuity of $f^*_{ijp}$ and $g^*_{ij}$ we conclude that
		\begin{eqnarray*}
			S_i&=&\lim_{q\to\infty} S_{i,q}=\frac{1}{1-\hat{a_i}}\sum_{j=1}^n\left(\sum_{p=1}^P \hat{b}_{ijp} f^*_{ijp}\left(\lim_{q\to\infty}S_{j,q-1}\right)  +  \hat{c}_{ij} g^*_{ij} \left(\lim_{q\to\infty}S_{j,q-1}\right)\right)\\
			&=&\frac{1}{1-\hat{a_i}}\sum_{j=1}^n\left(\sum_{p=1}^P \hat{b}_{ijp} f^*_{ijp}\left(S_{j}\right)  +  \hat{c}_{ij} g^*_{ij} \left(S_{j}\right)\right).
		\end{eqnarray*}
		\\
		Proof of item {\it ii.}\\
		
		For $\sigma\in\en_0$ and $\ov\psi\in\mathcal{B}^n$, consider $\ov{x}(m)=(x_1(m),\ldots,x_n(m))=\ov{x}(m,\sigma,\ov\psi)$ the solution of \eqref{model-principal}-\eqref{model-principal-IC}. Item {\it ii.} of Lemma \ref{lem-solucoes-limitadas} assures that $\ov{x}(m)$ is bounded, thus we may define
		$$
		M_i:=\dst\limsup_{m\to\infty}|x_i(m)|\in\er,\quad i\in[1,n]_\ze.
		$$
		For each $i\in[1,n]_\ze$ also define $S_{i,0}:=\dst\frac{1}{1-a_i^+}\sum_{j=1}^{n}\left(\sum_{p=1}^P b^+_{ijp} F_{ijp}M_j +  c^+_{ij} G_{ij} M_j\right)$.
		
		Given $\varepsilon>0$ there is $m_0\in[\sigma,\infty)_\ze$ such that
		$$
		|x_i(m)|\leq M_i+\varepsilon,\quad i\in[1,n]_\ze,\,m\in[m_0,\infty)_\ze.
		$$
		Arguing as in the previous item, we conclude an estimation analogous to \eqref{etimativa-para-item-ii}, i.e. there is $\mathfrak{m}\in[m_0,\infty)_\ze$ such that, for all $i\in[1,n]_\ze$ and $m\in[\mathfrak{m},\infty)_\ze$,
		\begin{eqnarray*}
			|x_i(m+1)| 
			&\leq&|a_i(m)||x_i(m-\nu_i(m))|+ \displaystyle\sum_{j=1}^{n}\sum_{p=1}^P |b_{ijp}(m)|f^*_{ijp}\left(M_j+\varepsilon\right) \nonumber\\ 
			& & + \displaystyle\sum_{j=1}^{n} |c_{ij} (m)|  \left(g^*_{ij} \left(M_j+\varepsilon\right)+\varepsilon G_{ij}\right) + |I_i(m)|.
		\end{eqnarray*}
		Consequently,
		\begin{eqnarray*}
			|x_i(m+1)| 
			&\leq&a_i^+|x_i(m-\nu_i(m))|+ \displaystyle\sum_{j=1}^{n}\sum_{p=1}^P b_{ijp}^+f^*_{ijp}\left(M_j+\varepsilon\right) \nonumber\\ 
			& & + \displaystyle\sum_{j=1}^{n} c_{ij}^+  \left(g^*_{ij} \left(M_j+\varepsilon\right)+\varepsilon G_{ij}\right) + |I_i(m)|,
		\end{eqnarray*}
		and applying again the $\limsup$ in both sides and taking $\varepsilon\to0^+$, we obtain
		$$
		M_i=\limsup_{m\to\infty}|x_i(m)|\leq \frac{1}{1-a_i^+} \displaystyle\sum_{j=1}^{n}\left(\sum_{p=1}^P b^+_{ijp} f^*_{ijp}\left(M_j\right)  +  c^+_{ij} g^*_{ij} \left(M_j\right)\right)=:S_{i,1},\quad i\in[1,n]_\ze.
		$$
		As $f_{ijp}$ and $g_{ij}$ satisfy (H6*), then $f^*_{ijp}$ and $g^*_{ij}$ also satisfy (H6*), and consequently $S_{i,1}\leq S_{i,0}$ for all $i\in[1,n]_\ze$.
		
		Finally, by repeating the arguments presented in the proof of the previous item {\em i.}, we conclude the existence of $\ov{S}=(S_1,\ldots,S_n)\geq0$ such that $\dst\limsup_{m\to\infty}|x_i(m)|\leq S_i$, for all $i\in[1,n]_\ze$, and \eqref{equ-aplicar-m-matriz-irredutivel} holds.
	\end{proof}
	
	Now we are in position to prove the main result of this paper. 
	
	\begin{teorema}\label{teo-principal}
		Assume hypotheses (H1)-(H5).
		\begin{enumerate}
			\item If (H6) holds and $\hat{\mathcal{M}}$ is an $M-$matrix, then the solution $\ov{x}(\cdot,\sigma,\ov{\psi})$ of \eqref{model-principal}-\eqref{model-principal-IC} satisfies
			$$
			\lim_{m\to\infty}|\ov{x}(m,\sigma,\ov\psi)|=0;
			$$
			\item If (H6*) holds and $\mathcal{M}^+$ is a singular irreducible $M-$matrix, then the solution $\ov{x}(\cdot,\sigma,\ov{\psi})$ of \eqref{model-principal}-\eqref{model-principal-IC} satisfies
			$$
			\lim_{m\to\infty}|\ov{x}(m,\sigma,\ov\psi)|=0.
			$$
		\end{enumerate}
	\end{teorema}
	\begin{proof}
		Here we present only the proof of item {\it i.}. The proof of item {\it ii.} follows similarly putting the matrix $\mathcal{M}^+$ instead of $\hat{\mathcal{M}}$. Recall the definitions of the matrices $\mathcal{M}^+$ and $\hat{\mathcal{M}}$ in \eqref{def-matriz-M+} and \eqref{def-matriz-capeu}, respectively.
		
		For $\sigma\in\en_0$ and $\ov\psi\in\mathcal{B}^n$, consider the solution $\ov{x}(m)=\ov{x}(m,\sigma,\ov\psi)=(x_1(m),\ldots,x_n(m))$ of  \eqref{model-principal}-\eqref{model-principal-IC}.
		
		From item {\it i.} of Lemma \ref{lem-sistema-equacoes}, there is $\ov{S}=(S_1,\ldots,S_n)\geq0$ such that $\dst\limsup_{m\to\infty}|x_i(m)|\leq S_i$ and \eqref{equ-aplicar-m-matriz} holds.
		
		To conclude the proof it is enough to prove that $S_i=0$ for all $i\in[1,n]_\ze$.
		
		First, we assume that $\ov{S}>0$. As $f^*_{ijp}$ and $g^*_{ij}$ also satisfy  (H6), then $f^*_{ijp}(S_j)<F_{ijp}S_j$ and $g_{ij}^*(S_j)<G_{ij}S_j$ for all $i,j\in[1,n]_\ze$ and $p\in[1,P]_\ze$. Consequently, from \eqref{equ-aplicar-m-matriz}, we obtain
		$$
		S_i=\frac{1}{1-\hat{a}_i}\sum_{j=1}^n\left(\hat{c}_{ij}g_{ij}^*(S_j)+\sum_{p=1}^P\hat{b}_{ijp}f_{ijp}^*(S_j)\right)<\frac{1}{1-\hat{a}_i}\sum_{j=1}^n\left(\hat{c}_{ij}G_{ij}+\sum_{p=1}^P\hat{b}_{ijp}F_{ijp}\right)S_j,
		$$ 
		for all $i\in[1,n]_\ze$, thus $\hat{\mathcal{M}}\ov{S}^T<0$ with $\ov{S}=(S_1,\ldots,S_n)>0$. Considering $\ov{z}=(z_1,\ldots,z_n)=(S_1,\ldots,S_n)>0$ and $\ov{y}=(y_1,\ldots,y_n)=\left(\hat{\mathcal{M}}\ov{S}^T\right)^T$, we have
		$$
		z_iy_i=S_i\left(S_i(1-\hat{a}_i)-\sum_{j=1}^n\left(\hat{c}_{ij}G_{ij}+\sum_{p=1}^P\hat{b}_{ijp}F_{ijp}\right)S_j\right)<0,\quad i\in[1,n]_\ze,
		$$
		which is not possible because of $\hat{\mathcal{M}}$ being an $M-$matrix (see Theorem \ref{teorema-B+P-1}).
		
		Now, suppose that there is $i\in[1,n]_\ze$ such that $S_i=0$ and $S_j>0$ for all $j\neq i$. We do not lose generality if we assume that $j=n$. Thus we have $S_n=0$ and $S_j>0$ for $j\in[1,n-1]_\ze$. Consequently, from \eqref{equ-aplicar-m-matriz} we obtain
		$$
		S_i=\frac{1}{1-\hat{a}_i}\sum_{j=1}^{n-1}\left(\hat{c}_{ij}g_{ij}^*(S_j)+\sum_{p=1}^P\hat{b}_{ijp}f_{ijp}^*(S_j)\right),\quad i\in[1,n-1]_\ze,
		$$ 
		and by the same argument as before we get another contradiction thus there is $j\in[1,n-1]_\ze$ such that $S_j=0$. 
		
		After $n-1$ steps we conclude that $S_2=\cdots=S_{n-1}=S_n=0$. From \eqref{equ-aplicar-m-matriz} we have
		$$
		S_1=\frac{1}{1-\hat{a}_1}\left(\hat{c}_{11}g^*_{11}(S_1)+\sum_{p=1}^P\hat{b}_{11p}f^*_{11p}(S_1)\right).
		$$
		
		
		If $S_1>0$, then 
		$$
		S_1=\frac{1}{1-\hat{a}_1}\left(\hat{c}_{11}g^*_{11}(S_1)+\sum_{p=1}^P\hat{b}_{11p}f^*_{11p}(S_1)\right)<\frac{1}{1-\hat{a}_1}\left(\hat{c}_{11}G_{11}S_1+\sum_{p=1}^P\hat{b}_{11p}F_{11p}S_1\right),
		$$
		thus
		$$
		S_1\left(1-\hat{a}_1-\hat{c}_{11}G_{11}-\sum_{p=1}^P\hat{b}_{11p}F_{11p}\right)<0,
		$$
		which is not possible because of Theorem \ref{teorema-B+P-1} again. Consequently $S_1=0$ and the proof is concluded.
	\end{proof}
	
	Let us now take the following autonomous discrete-time Hopfield model with infinite delays
	\begin{eqnarray}\label{model-autonomo}
		x_i(m+1) &=&a_ix_i(m-\nu_i)+ \displaystyle\sum_{j=1}^{n}\sum_{p=1}^P b_{ijp}f_{ijp}\left(x_j(m-\tau_{ijp})\right) \nonumber\\ 
		& & + \displaystyle\sum_{j=1}^{n} c_{ij}   \sum_{l=0}^{\infty} \zeta_{ijl} g_{ij} (x_j(m-l)) + I_i,\,\,\quad i\in[1,n]_\ze,\,m\in\en_0,
	\end{eqnarray}
	where $n,P\in\en$ and, for each $i,j\in[0,n]_\ze$ and $p\in[0,P]_\ze$, $a_i\in(-1,1)$, $\nu_i,\tau_{ijp}\in\en_0$, $\zeta_{ijl},b_{ijp},c_{ij},I_i\in\er$, and $f_{ijp},g_{ij}:\er\to\er$ are functions. 
	
	As an immediate  consequence of Theorem \ref{teo-principal}, we derive the following global stability criteria for the autonomous discrete-time Hopfield neural network model \eqref{model-autonomo}.

	\begin{corolario}\label{criterio-melho-hong+ma}
		Assume that $I_i=0$ for all $i\in[0,n]_\ze$ and that hypothesis (H4) holds.
		
		Defining $\mathcal{D}:=diag(1-|a_1|,\ldots,1-|a_n|)$, consider the matrix 
		\begin{eqnarray}\label{matrix-situacao-autonoma}
			\mathcal{M}:=\mathcal{D}-\left[|c_{ij}|G_{ij}+\sum_{p=1}^P|b_{ijp}|F_{ijp}\right]_{i,j=1}^n.
		\end{eqnarray}
		\begin{enumerate}
			\item If (H6) holds and $\mathcal{M}$ is an $M-$matrix, then the zero solution of \eqref{model-autonomo}  is globally attractive, that is for all $\sigma\in\en_0$ and $\ov{\psi}\in\mathcal{B}^n$, the solution $\ov{x}(\cdot,\sigma,\ov{\psi})$ of \eqref{model-autonomo} satisfies 
			$$
			\lim_{m\to\infty}|\ov{x}(m,\sigma,\ov\psi)|=0;
			$$
			\item If (H6*) holds and $\mathcal{M}$ is a singular irreducible $M-$matrix, then the zero solution of \eqref{model-autonomo}  is globally attractive, that is for all $\sigma\in\en_0$ and $\ov{\psi}\in\mathcal{B}^n$, the solution $\ov{x}(\cdot,\sigma,\ov{\psi})$ of \eqref{model-autonomo} satisfies
			$$
			\lim_{m\to\infty}|\ov{x}(m,\sigma,\ov\psi)|=0.
			$$
		\end{enumerate}
		
	\end{corolario}

	By taking $\nu_i=0$, $P=1$, $c_{ij}=0$, $I_i=0$, and $f_{ij}=f_{j}$ in \eqref{model-autonomo} for all $i,j\in[1,n]_\ze$, we obtain the next discrete-time Hopfield model with finite delays:
	\begin{eqnarray}\label{model-hong+ma}
		x_i(m+1) =a_ix_i(m)+ \displaystyle\sum_{j=1}^{n} b_{ij}f_{j}\left(x_j(m-\tau_{ij})\right),\,\,\quad i\in[1,n]_\ze,\,m\in\en_0,
	\end{eqnarray}
	and the matrix $\mathcal{M}$ defined in \eqref{matrix-situacao-autonoma} takes the form
	$$
	\mathcal{L}:=\mathcal{D}-\big[|b_{ij}|F_{j}\big]_{i,j=1}^n,
	$$
	where $\mathcal{D}=diag(1-|a_1|,\ldots,1-|a_n|)$.
	We note that model \eqref{model-hong+ma} has finite delays. Consequently, it is possible to transform it into a non-delay discrete-time model, making it easier to handle than \eqref{model-autonomo}.
	\begin{rem}
		\emph{In \cite[Theorem 3.1]{hong+ma}, the global attractivity of \eqref{model-hong+ma} is established under the assumption that $\mathcal{L}$ is an $M-$matrix and the activation functions, $f_j:\er\to\er$, are differentiable, satisfying the conditions
			$$
			f_j(0)=0,\quad \lim_{x\to\infty}f_j(x)=1,\quad\lim_{x\to-\infty}f_j(x)=-1,
			$$
			and 
			\begin{eqnarray}\label{hip-hong-ma}
				f'_j(0)=1>f'_j(x)>0,\quad x\in\er\setminus\{0\}.
			\end{eqnarray} 
			These conditions are more restrictive than hypothesis (H6). Additionally  model \eqref{model-autonomo} includes both infinite delays and delays in the leakage terms. Together, theses factors demonstrate that  item {\it i.} in Corollary  \ref{criterio-melho-hong+ma} provides a significant improvement over the global criterion presented in  \cite{hong+ma}. We note that the authors in \cite{hong+ma} wrote the hypothesis $f'_j(0)=\sup_{x\in\er}f_j'(x)=1$ instead of \eqref{hip-hong-ma}. However, this is inaccurate because \eqref{hip-hong-ma} is need to obtain $\bar{f}_j(M_j)<M_j$ in \cite[page 4940, line 5]{hong+ma}.}
		
		\emph{Finally, we note that, to the best of our knowledge, stability criteria involving singular $M-$matrix along with unbounded activation functions have not been previously studied. Consequently, items {\it ii.} in Theorem \ref{teo-principal} and Corollary \ref{criterio-melho-hong+ma} give answers to these situations.}\\
	\end{rem}
	
	Now, taking $P=2$, $\tau_{ij1}=0$, $I_i=0$, and $f_{ij1}=f_{ij2}=g_{ij}=g_j$ in \eqref{model-autonomo} for all $i,j\in[1,n]_\ze$, we obtain the following model
	\begin{eqnarray}\label{model-autonomo-jj2022}
		x_i(m+1) &=&a_ix_i(m-\nu_i)+ \displaystyle\sum_{j=1}^{n} b_{ij1}g_{j}\left(x_j(m)\right)+\sum_{j=1}^{n} b_{ij2}g_{j}\left(x_j(m-\tau_{ij2})\right) \nonumber\\ 
		& & + \displaystyle\sum_{j=1}^{n} c_{ij}   \sum_{l=0}^{\infty} \zeta_{ijl} g_{j} (x_j(m-l)),\,\,\quad i\in[1,n]_\ze,\,m\in\en_0,
	\end{eqnarray}
	and the matrix $\mathcal{M}$ defined in \eqref{matrix-situacao-autonoma} has the form
	$$
	\mathcal{E}:=\mathcal{D}-\left[\big(|b_{ij1}|+|b_{ij2}|+|c_{ij}|\big)G_j\right]_{i,j=1}^n,
	$$
	where $\mathcal{D}=diag(1-|a_1|,\ldots,1-|a_n|)$.
	Assume that $a_i\in(-1,1)$ for all $i\in[0,n]_\ze$, and that hypotheses (H4) and (H6*) hold. From Corollary \ref{criterio-melho-hong+ma}, we conclude that if $\mathcal{E}$ is a singular irreducible $M-$matrix, then the zero solution of \eqref{model-autonomo-jj2022} is globally attractive within the set of solutions with bounded initial conditions.\\
	
	\begin{rem}
		\emph{It is worth noting that model \eqref{model-autonomo-jj2022} is also a particular case of a discrete-time Hopfield model studied in \cite[model (9)]{jj-jdea}. Under the additional condition $\dst\sum_{l=0}^\infty\Ne^{\xi l}|\zeta_{ijl}|<\infty$,  \cite[Theorem 3.2]{jj-jdea} allows us to conclude the global exponential stability of the zero solution of \eqref{model-autonomo-jj2022}, provided that $\mathcal{E}$ is a non-singular $M-$matrix.}\\
	\end{rem}
	
	Now, we present a global attractivity  criterion for the case where the activation functions in \eqref{model-autonomo} are Lipschitz functions. 
	
	\begin{corolario}
		Assume that \eqref{model-autonomo} has an equilibrium point $\ov{x}^* \in \er^n$, and that hypotheses (H4) and
		\begin{description}
			\item[(H6$^\dag$)] For each $i,j \in [1,n]_\ze$ and $p \in [1,P]_\ze$, there exist constants $F_{ijp}, G_{ij} > 0$ such that
			$$
			|f_{ijp}(u) - f_{ijp}(v)| < F_{ijp} |u - v| \quad \text{and} \quad
			|g_{ij}(u) - g_{ij}(v)| < G_{ij} |u - v|,
			$$
			for $u, v \in \er$ with $u \neq v$,
		\end{description}
		hold.
		
		If $\mathcal{M}$, defined in \eqref{matrix-situacao-autonoma}, is a singular irreducible $M$-matrix, then for any $\sigma \in \en_0$ and $\ov{\psi} \in \mathcal{B}^n$, the solution $\ov{x}(\cdot, \sigma, \ov{\psi})$ of \eqref{model-autonomo} satisfies
		$$
		\lim_{m \to \infty} |\ov{x}(m, \sigma, \ov{\psi}) - \ov{x}^*| = 0.
		$$
	\end{corolario}
	\begin{proof}
		Assume that $\ov{x}^*=(x_1^*,\ldots,x_n^*)\in\er^n$ is an equilibrium point of \eqref{model-autonomo}. Consequently,
		\begin{eqnarray*}\label{model-autonomo-equilibrio}
			x^*_i =a_ix^*_i+ \displaystyle\sum_{j=1}^{n}\sum_{p=1}^P b_{ijp}f_{ijp}\left(x^*_j\right)  + \displaystyle\sum_{j=1}^{n} c_{ij}   \sum_{l=0}^{\infty} \zeta_{ijl} g_{ij} (x^*_j) + I_i,\,\,\quad i\in[1,n]_\ze.
		\end{eqnarray*}
		Let $\sigma\in\en_0$ and $\ov{\psi}\in\mathcal{B}^n$. Consider $\ov{x}(m)=\ov{x}(m,\sigma,\ov{\psi})$ the solution of \eqref{model-autonomo} with initial condition $\ov{x}_\sigma=\ov{\psi}$.
		
		The change of variables $\ov{z}(m)=(z_1(m),\ldots,z_n(m))=\ov{x}(m)-\ov{x}^*$ transforms \eqref{model-autonomo} into
		\begin{eqnarray}\label{model-principal-lipschitz}
			z_i(m+1) &=&a_iz_i(m-\nu_i)+ \displaystyle\sum_{j=1}^{n}\sum_{p=1}^P b_{ijp}\tilde{f}_{ijp}\left(z_j(m-\tau_{ijp})\right) \nonumber\\ 
			& & + \displaystyle\sum_{j=1}^{n} c_{ij}  \sum_{l=0}^{\infty} \zeta_{ijl} \tilde{g}_{ij} (z_j(m-l)),\,\,\quad i\in[1,n]_\ze,\,m\in\en_0,
		\end{eqnarray}
		where $\tilde{f}_{ijp},\tilde{g}_{ij}:\er\to\er$ are defined by
		$
		\tilde{f}_{ijp}(u)=f_{ijp}(u+x^*_j)-f_{ijp}(x^*_j)$ and $\tilde{g}_{ij}(u)=g_{ij}(u+x^*_j)-g_{ij}(x^*_j)$, for $u\in\er$, respectively.
		
		From (H6$^\dag$), we obtain $|\tilde{f}_{ijp}(u)|< F_{ijp}|u|$ and $|\tilde{g}_{ij}(u)|< G_{ij}|u|$ for all $u\in\er$, thus (H6*) holds. Applying item {\it ii.} in Corollary \ref{criterio-melho-hong+ma} to the system \eqref{model-principal-lipschitz} we conclude that
		$$
		\lim_{m\to\infty}|\ov{z}(m)|=0,
		$$
		which means that
		$$
		\lim_{m\to\infty}|\ov{x}(m,\sigma,\ov{\psi})-\ov{x}^*|=0.
		$$ 
	\end{proof}

	\section{Numerical Examples}\label{numerical-examples}
	
	Here, we provide two numerical examples to highlight the effectiveness of the new results given in this work.

	\exemplo  In model  \eqref{model-principal} we take  $n=3$, $P=2$, the following coefficients in the leakage terms
	\begin{align*}
		a_1(m)&= \frac{1}{6}\cos (\pi m) &a_2(m) &=\frac{1}{m+1}, &a_3(m)&=\frac{1}{2},
	\end{align*}
	the delays in the leakage terms
	\begin{align*}
		\nu_1(m)&=1, &nu_2(m)&=2, &nu_3(m)&=3,
	\end{align*}
	the coefficients in the terms with discrete time-varying delays
	\begin{align*}
		b_{111}(m)&=0,&b_{112}(m)&=0,&b_{121}(m)&= \frac{1}{6}\sin\left(\frac{\pi}{2} m\right),\\
		b_{131}(m)&=0,&b_{132}(m)&=0,&b_{122}(m)&=\frac{1}{6}\cos(\pi m),\\
		b_{211}(m)&=0,&
		b_{212}(m)&=0,&b_{221}(m)&=\frac{1}{3}\sin\left(\frac{\pi}{2} m\right),\\
		b_{222}(m)&=0,&	b_{231}(m)&=0,&b_{232}(m)&=\frac{1}{6}\cos(\pi m),\\
		b_{311}(m)&=0,&b_{312}(m)&=0,&b_{321}(m)&=0,\\
		b_{322}(m)&=0,&b_{332}(m)&=0,&b_{331}(m)&=\frac{1}{2(1+\Ne^{-m})}-\frac{1}{4},
	\end{align*}
	the coefficients in the terms with distributed delays
	\begin{align*}
		c_{12}(m)&=0,&c_{13}(m)&=0,&c_{11}(m)&=\frac{2}{3}\tanh(m)\\
		c_{21}(m)&=0,&c_{22}(m)&=0,&c_{23}(m)&=\frac{2}{3}\cos(\pi m),\\
		c_{31}(m)&=0,&c_{32}(m)&=0,&c_{33}(m)&=\cos(\pi m),\\
	\end{align*}
	and the discrete time-varying delays
	\begin{align*}
		\tau_{121}(m)&=\left\lfloor\frac{m}{4}\right\rfloor,&\tau_{122}(m)&=1,&\tau_{221}(m)=2,\\
		\tau_{232}(m)&=3,&\tau_{331}(m)&=3,
	\end{align*}
	recalling that $\lfloor x \rfloor$ denotes the integer part of the real number $x\in\er$. We also take
	the activation functions
	\begin{align*}
		f_{121}(u)&=f_{122}(u)=\min\left\{\arctan(|u|),1\right\},\\
		f_{221}(u)&=f_{232}(u)=\tanh(u),\ \ \ \
		f_{331}(u)=\frac{u}{\sqrt{1+u^2}},\\
		g_{11}(u)&=g_{23}(u)=g_{33}(u)=\frac{1}{1+\Ne^{-u}}-\frac{1}{2},
	\end{align*}
	the Kernel factors in the terms with infinite distributed delays
	\begin{align*}
		\zeta_{11l}&=\frac{1}{(l+1)(l+2)}, &\zeta_{23l}=\zeta_{33l}=\frac{1}{2^{l+1}},\\
	\end{align*} 
	for $l\in\en_0$, and the external inputs 
	\begin{align*}	
		I_1(m)&=\frac{1}{(m+1)^2},&I_2(m)&=\frac{1}{2^m},&I_3(m)&=\frac{1}{(m+1)^5}.
	\end{align*}
	
	In this example we have $F_{ijp}=1$, $G_{ij}=\frac{1}{4}$ for $i,j=1,2,3$ and $p=1,2$, and the matrix $\hat{\mathcal{M}}$ in \eqref{def-matriz-capeu} has the form
	$$
	\hat{\mathcal{M}}=\left[\begin{array}{ccc}
		\frac{5}{6} & 0 & 0\\
		\\
		0 & 1 & 0\\
		\\
		0 & 0 &\frac{1}{2}\\
	\end{array}\right]-\left[\begin{array}{ccc}
		\frac{1}{6} & \frac{1}{3} & 0\\
		\\
		0 & \frac{1}{3} & \frac{1}{3}\\
		\\
		0 & 0 & \frac{1}{2}\\
	\end{array}\right]=\left[\begin{array}{ccc}
		\frac{2}{3} & -\frac{1}{3} & 0\\
		\\
		0 & \frac{2}{3} & -\frac{1}{3}\\
		\\
		0 & 0 & 0\\
	\end{array}\right], 
	$$
	which is a singular M-matrix. As the hypothesis (H1)-(H6) hold, item {\it i.} of Theorem \ref{teo-principal} ensures that all solutions, $\ov{x}(m, \sigma, \ov{\psi})$ with $\sigma\in\en_0$ and $\ov\psi\in\mathcal{B}^3$, of this example converge to $(0,0,0)$ as $m$ tends to infinity. Figure \ref{fig:modelacaox2} shows the plot of the solution  $\ov{x}(m)=\ov{x}(m,0,\ov{\psi})=(x_1(m),x_2(m),x_3(m))$ of this illustrative numerical example with initial condition $\ov{x}_0(s)=
	\overline{\psi}(s)=\left\{\begin{array}{ll}
		\left(\sin(s),-\cos(s),\Ne^s\right),& s\in[-9,0]_\ze\\
		(0,0,0),&s\in(-\infty,-9)_\ze
	\end{array}\right.$.
	\begin{figure}[h]
		\centering
		{\includegraphics[width=15.5cm]{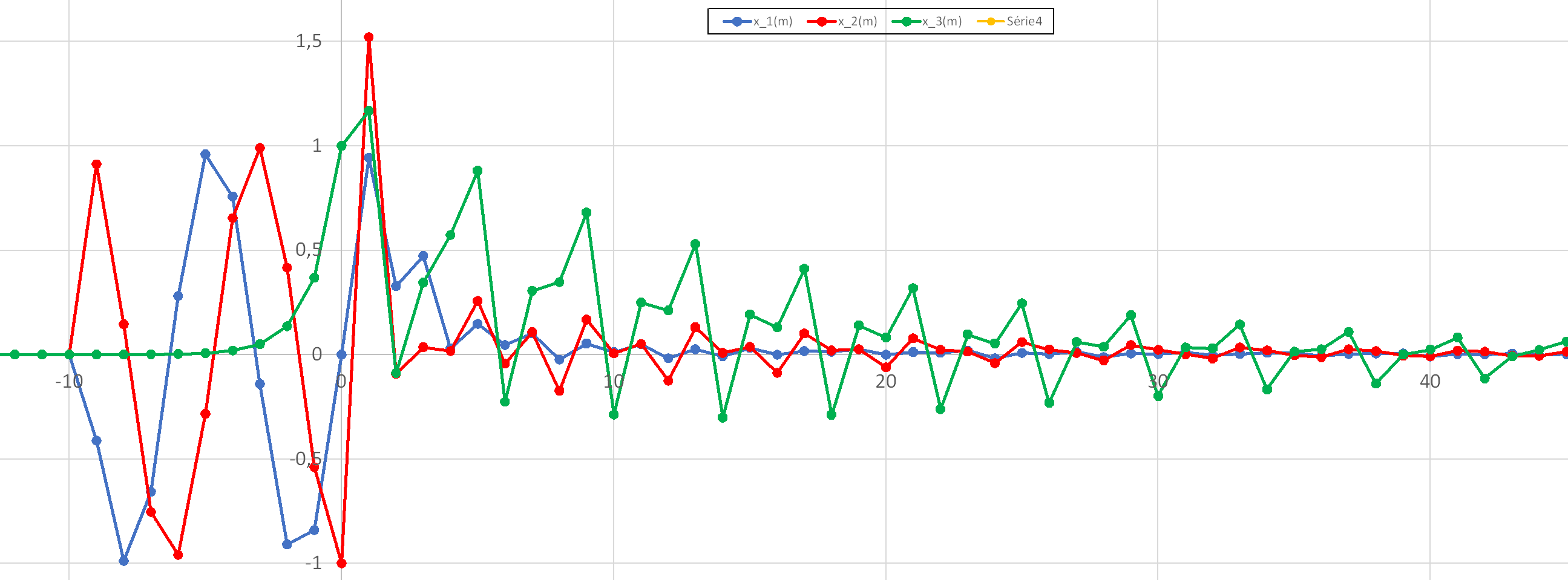}}
		\caption{Plot of a solution $\ov{x}(m)=(x_1(m),x_2(m),x_3(m))$ of the numerical Example 4.1..
		}\label{fig:modelacaox2}
	\end{figure}
	
	We note that this numerical example is non-autonomous and has unbounded delays. Therefore, the stability criterion established in \cite{hong+ma} cannot be applied in this case.

	\exemplo In this second numerical example consider, in model  \eqref{model-principal}, $n=3$, $P=2$, the coefficients in the leakage terms
	\begin{align*}
		a_1(m)&=\frac{2}{3}\cos (\pi m),&a_2(m)&=\frac{1}{3}\sin \left(\frac{\pi}{2} m\right),\\
		a_3(m)&=\frac{2}{3}\sin \left(\frac{\pi}{2} m\right),
	\end{align*}
	the delays in the leakage terms
	\begin{align*}
		\nu_1(m)&=1,&\nu_2(m)=&3,&nu_3(m)&=4,
	\end{align*}
	the coefficients in the terms with discrete time-varying delays
	\begin{align*}
		b_{111}(m)&=0,&b_{112}(m)&=0,&b_{121}(m)&= \frac{1}{9},\\
		b_{122}(m)& =\frac{1}{9}\cos(\pi m),&b_{131}(m)&=0,&b_{132}(m)&=0,\\
		b_{211}(m)&=\frac{1}{9}\sin \left(\frac{\pi}{2} m\right),&b_{212}(m)&=0,&b_{221}(m)&=0,\\
		b_{222}(m)&=0,&b_{231}(m)&=\frac{1}{4},&b_{232}(m)&=0,\\
		b_{311}(m)&=0,&b_{312}(m)&=0,&b_{322}(m)&=0,\\ 
		b_{321}(m)&=\frac{1}{12}\cos(\pi m),&b_{331}(m)&=0,&b_{332}(m)&=0,  
	\end{align*}
	the coefficients in the terms with distributed delays
	\begin{align*}
		c_{11}(m)&=0,&c_{12}(m)&=\frac{1}{9},&c_{13}(m)&=0,\\
		c_{21}(m)&=\frac{2}{9},&c_{22}(m)&=0,&c_{23}(m)&=\frac{1}{12},\\
		c_{31}(m)&=0,&c_{32}(m)&=\frac{1}{4}\sin\left(\frac{\pi}{2} m\right),&c_{33}(m)&=0,
	\end{align*}
	the discrete time-varying delays
	\begin{align*}
		\tau_{121}(m)&=0,&\tau_{122}(m)&=\left\lfloor\frac{m}{3}\right\rfloor,&\tau_{211}(m)&=1,\\
		\tau_{231}(m)&=5,&\tau_{321}(m&)=\left\lfloor\frac{m}{4}\right\rfloor,	
	\end{align*}
	the activation functions
	\begin{align*}
		f_{121}(u)&=\arctan(u),\\ 
		f_{122}(u)&=f_{211}(u)=\max\left\{\tanh(|u|),u-\frac{1}{10}\right\},\\
		f_{231}(u)&=f_{321}(u)=\tanh(u),\\
		g_{12}(u)&=g_{21}=g_{23}(u)=g_{32}(u)=\tanh(u)	
	\end{align*}
	the Kernel factors in the terms with infinite distributed delays
	\begin{align*}
		\zeta_{12l}=\zeta_{21l}=\zeta_{23l}=\zeta_{32l}=\frac{1}{2^{l+1}},
	\end{align*}
	for $l\in\en_0$, and the external inputs
	\begin{align*}	
		I_1(m)&=\frac{1}{(m+1)^2},&I_2(m)&=\frac{1}{2^m},&I_3(m)=\frac{1}{(m+1)^3}.
	\end{align*}
	
	We have $F_{ijp}=G_{ij}=1$ for $i,j=1,2,3$ and $p=1,2$, and the matrix $\mathcal{M}^+$ in \eqref{def-matriz-M+} has the form
	$$
	\mathcal{M}^+=\left[\begin{array}{ccc}
		\frac{1}{3} & 0 & 0\\
		\\
		0 & \frac{2}{3} & 0\\
		\\
		0 & 0 &\frac{1}{3}\\
	\end{array}\right]-\left[\begin{array}{ccc}
		0 & \frac{1}{3} & 0\\
		\\
		\frac{1}{3} & 0 & \frac{1}{3}\\
		\\
		0 & \frac{1}{3} & 0\\
	\end{array}\right]=\left[\begin{array}{ccc}
		\frac{1}{3} & -\frac{1}{3} & 0\\
		\\
		-\frac{1}{3} & \frac{2}{3} & -\frac{1}{3}\\
		\\
		0 & -\frac{1}{3} & \frac{1}{3}\\
	\end{array}\right],
	$$
	\begin{figure}[h]
		\centering
		{\includegraphics[width=15.5cm]{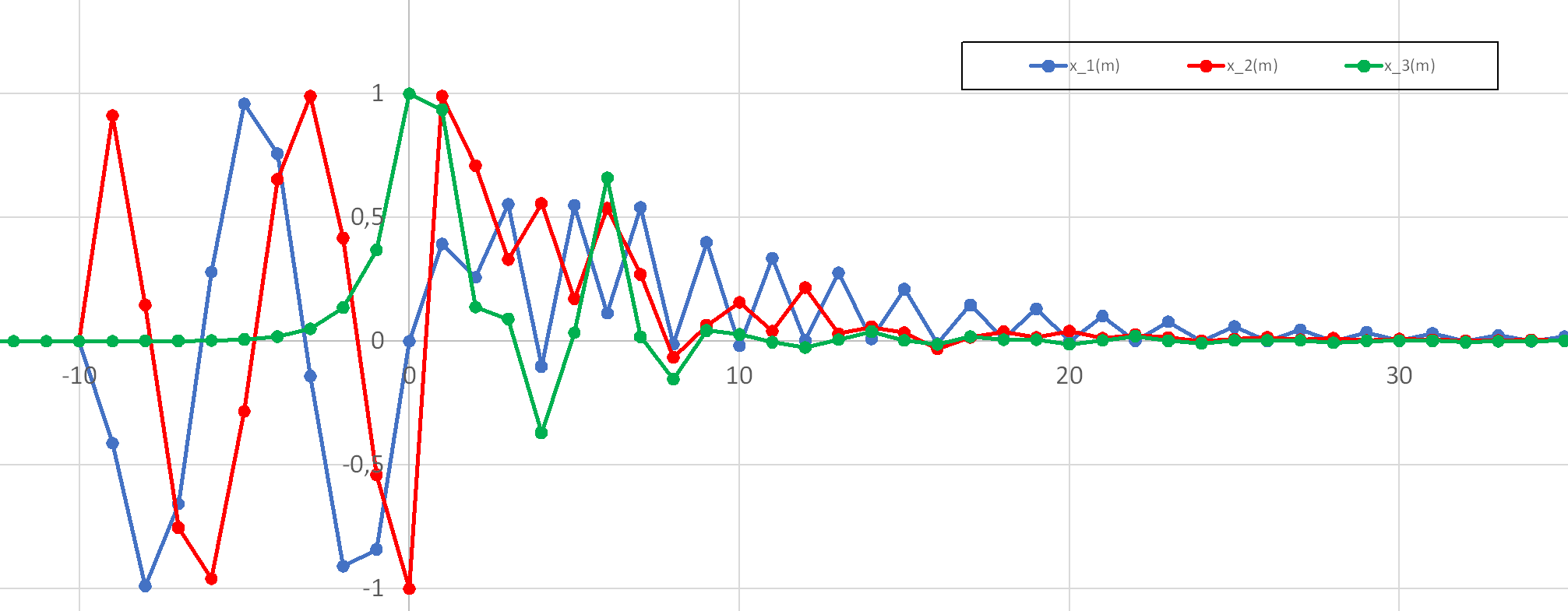}}
		\caption{Plot of a solution $\ov{x}(m)=(x_1(m),x_2(m),x_3(m))$ of the numerical Example 4.2..
		}\label{fig:modelacaox}
	\end{figure}
	which is a singular irreducible $M-$matrix. As the hypotheses (H1)-(H5) and (H6*) hold, item {\it ii.} of Theorem \ref{teo-principal} ensures that all solutions, $\ov{x}(m, \sigma, \ov{\psi})$ with $\sigma\in\en_0$ and $\ov\psi\in\mathcal{B}^3$, of this example converge to $(0,0,0)$ as $m$ tends to infinity. Figure \ref{fig:modelacaox} shows the plot of the solution $\ov{x}(m)=\ov{x}(m,0,\ov{\psi})=(x_1(m),x_2(m),x_3(m))$ of this illustrative numerical example with initial condition $\ov{x}_0(s)=
	\overline{\psi}(s)=\left\{\begin{array}{ll}
		\left(\sin(s),-\cos(s),\Ne^s\right),& s\in[-9,0]_\ze\\
		(0,0,0),&s\in(-\infty,-9)_\ze
	\end{array}\right.$. 
	
	We note that in this numerical example, there are unbounded delays and some activation functions are also unbounded. Therefore, the stability criterion established in \cite{hong+ma} also cannot be applied in this case.

	\section{Conclusions}
	
	In this work, we have established two global attractivity criteria for a discrete-time non-autonomous Hopfield neural network model with infinite distributed and time-varying delays (Theorem \ref{teo-principal}). Notably, unlike the common assumptions in the literature, the stability criteria presented here involve $M$-matrices that are not necessarily invertible.
	
	The first main stability criterion, outlined in item {\it i.} of Theorem \ref{teo-principal}, assumes that the activation functions of the model are bounded and sublinear. The second main stability criterion, presented in item {\it ii.} of Theorem \ref{teo-principal}, requires only that the activation functions be sublinear, with the additional assumption that the $M$-matrix is singular and irreducible. 
	
	The two numerical simulations provided illustrate the practical implications and insights offered by this work.
	
	

	\noindent
	{\bf Acknowledgments.}\\
	The research was partially financed by Portuguese Funds through FCT (Funda\c{c}\~ao para a Ci\^encia e a Tecnologia) within the Projects UIDB/00013/2020 and UIDP/00013/2020 of CMAT-UM.

\end{document}